\documentclass[a4paper, 11pt, reqno]{amsart}

\usepackage[T1]{fontenc}
\usepackage[latin1]{inputenc}
\usepackage[margin=3cm]{geometry}
\usepackage{amsmath}
\usepackage{amssymb}
\usepackage{amsthm}
\usepackage[colorlinks,urlcolor=blue]{hyperref}

\newcommand{\abs}[1]{\lvert #1 \rvert}
\newcommand{\bigabs}[1]{\bigl\lvert #1 \bigr\rvert}
\newcommand{\W}[2]{W^{(#1)}_{#2}}
\newcommand{\Ggraph}[3][n]{G_{#1,#2}^{(#3)}}
\newcommand{\Gprimegraph}[3][n]{G'^{(#3)}_{#1,#2}}

\newcommand{\yvec}{\mathbf{y}}

\newcommand{\N}{\mathbb{N}}
\newcommand{\R}{\mathbb{R}}

\newtheorem{thm}{Theorem}[section]
\newtheorem{theorem}[thm]{Theorem}
\newtheorem{claim}[thm]{Claim}

\newtheorem{cor}[thm]{Corollary}
\newtheorem{corollary}[thm]{Corollary}
\newtheorem{lem}[thm]{Lemma}
\newtheorem{lemma}[thm]{Lemma}

\newtheorem{problem}[thm]{Problem}
\newtheorem{prop}[thm]{Proposition}
\newtheorem{proposition}[thm]{Proposition}

\newtheorem{question}[thm]{Question}

\theoremstyle{remark}

\newtheorem{remark}[thm]{Remark}

\theoremstyle{definition}
\newtheorem{example}[thm]{Example}

\title{Strong Tur\'{a}n stability}

\author{Mykhaylo Tyomkyn}
\address{School of Mathematics, University of Birmingham, Edgbaston, Birmingham, B15 2TT, UK}
\email{\href{mailto:m.tyomkyn@bham.ac.uk}{m.tyomkyn@bham.ac.uk}}

\author{Andrew~J. Uzzell}
\address{Department of Mathematics, University of Nebraska, Lincoln NE 68588-0130, USA}
\email{\href{mailto:andrew.uzzell@math.unl.edu}{andrew.uzzell@math.unl.edu}}

\date{\today}

\thanks{The first author's research was supported by EPSRC grant EP/K033379/1.  This research was carried out while the second author was at Uppsala University, where he was supported by the Knut and Alice Wallenberg Foundation.}

\begin{document}

\begin{abstract}
We study the behaviour of $K_{r+1}$-free graphs~$G$ of almost extremal size, that is, typically, $e(G)=ex(n,K_{r+1})-O(n)$. We show that such graphs must have a large amount of `symmetry', in particular that all but very few vertices of~$G$ must have twins. As a corollary, we obtain a new, short proof of a theorem of Simonovits on the structure of extremal graphs with $\omega(G)\leq r$ and $\chi(G)\geq k$ for fixed $k \geq r \geq 2$.
\end{abstract}

\maketitle

\section{Introduction}\label{se:intro}

Let $T_{n,r}$ denote the Tur\'{a}n graph on $n$ vertices with $r$ partition classes of size $\lfloor n/r \rfloor$ or $\lceil n/r \rceil$ each, and put $t_{n,r}:=e(T_{n,r})$. From  Tur\'{a}n's theorem we know that $t_{n,r}$ maximises the size of a $K_{r+1}$-free graph of order~$n$. One of the best-known extensions of Tur\'{a}n's theorem is the Erd\H{o}s--Simonovits stability theorem, which says, in particular, that a $K_{r+1}$-free graph on $n$ vertices and $t_{n,r}-o(n^2)$ edges can be turned into $T_{n,r}$ by adding or removing $o(n^2)$ edges. 
To phrase it qualitatively, a $K_{r+1}$-free graph whose size is close to being extremal looks essentially like the extremal graph. This behaviour has become known as \emph{stability} and has been extensively studied in various structures.

In this paper we are concerned with different aspects of Tur\'{a}n stability. More concretely, we will study $K_{r+1}$-free graphs~$G$ with $e(G)=t_{n,r}-O(n)$ or $e(G)=t_{n,r}-O(n\log n)$. This is much closer to the Tur\'{a}n threshold than the range of the Erd\H{o}s--Simonovits stability theorem and allows to observe different aspects of stability. Our results can therefore be viewed as a part of a larger programme of studying the `phase transition' of $K_{r+1}$-free graphs near the Tur\'{a}n threshold that has been emphasized by Simonovits.

First, in Section \ref{se:strongstab} we give a new proof of a theorem on the maximum size of a $K_{r+1}$-free graph of chromatic number at least~$r + 1$.  This result was first explicitly proved by Brouwer~\cite{Brou}, although implicitly it follows from earlier work of Simonovits~\cite{Sim68}, and in the case $r=2$ it was proved by Andr\'asfai, Erd\H{o}s, and Gallai~\cite{Erd}. It has also been re-discovered several times~\cite{AFGS,HT,KangPik}. 

Let
\begin{equation}\label{eq:hdef}
	h(n, r) = \begin{cases}
	    	t_{n,r} - \bigl\lfloor \frac{n}{r} \bigr\rfloor + 1, & n \geq 2r + 1,\\
		t_{n,r} - 2, & r + 3 \leq n \leq 2r,
	    \end{cases}
\end{equation}
and note that the second case is vacuous if $r = 2$.

\begin{thm}\label{thm1}
If $n \geq r + 3$, then every $K_{r+1}$-free graph of order~$n$ and size at least~$h(n,r) + 1$ is $r$-colourable.
\end{thm}

Unlike the Erd\H{o}s--Simonovits theorem, which says that a $K_{r+1}$-free graph on sufficiently many edges is \emph{approximately} $r$-partite, this theorem gives a condition for a $K_{r+1}$-free graph to \emph{actually} be $r$-partite. 

A natural generalisation of Theorem~\ref{thm1} would be to find the maximal number of edges in a graph~$G$ with $|V(G)|=n$, $\omega(G)=r$ and $\chi(G)\geq k$. It is easy to see that the extremal number is of order~$t_{n,r}-O(n)$: for instance, take the disjoint union of a Tur\'an graph~$T_{n',r}$ and a finite size graph~$G'$ with $\omega(G')=r$ and $\chi(G')\geq k$. Determining the constant in the linear term asymptotically as $k\rightarrow \infty$ is less interesting in its own right, since it is closely related to the asymptotic behaviour of the Ramsey numbers~$R(r+1,k)$. (This connection is discussed further in Remark~\ref{re:TriangleFreeAsymptotics}.) We do, however, determine the constant exactly in the first open case $k=r+2$; see Theorem \ref{thm:rplus2}.

A much more interesting problem is the structure of the extremal graphs. One simple way to construct such graphs (more efficiently than the trivial construction given above) is the following: take a finite size graph $G'$ with $\omega(G')=r$ and $\chi(G')=k$, and blow up an $r$-clique of $G'$ in a way that maximises the number of edges. Let us call a graph (or, more precisely, a graph sequence) \emph{simple} if it is a blow-up of a bounded order graph. It is natural to ask whether the extremal graph must be simple. This was answered in the affirmative by Simonovits for $r=2$ in~\cite{Sim69} and (as a part of a more general result) for arbitrary $r$ in~\cite{Sim74}.

Recall that a graph~$G$ is called \emph{maximal $H$-free} or \emph{$H$-saturated} if it is $H$-free but adding any edge to~$G$ would create a copy of~$H$ as a subgraph. For $H=K_{r+1}$ the corresponding saturated graphs are also called \emph{$(r+1)$-saturated}.  In Section \ref{se:CliqueSat} we suggest a new generalisation of Theorem~\ref{thm1}, namely the study of $K_{r+1}$-saturated graphs on many edges; note that the extremal graph for a given chromatic number is a special case. In the spirit of Simonovits' theorem we prove sharp bounds on how large $e(G)$ should be in order for $G$ to be simple. Perhaps surprisingly, the thresholds for $r=2$ and for $r\geq 3$ turn out to be substantially different, with the proof being very short in the former case and more involved in the latter.

\begin{theorem}\label{thm:SimpleBound}
For every $c > 0$ every $3$-saturated graph~$G$ on $n$ vertices with $e(G)> t_{n,2}-cn$ is simple.

Let $r \geq 3$.  For every $\varepsilon>0$ every $(r+1)$-saturated graph~$G$ on $n$ vertices with $e(G)> t_{n,r}-(2-\varepsilon)n/r$ is simple.
\end{theorem}

Taking this study further, we obtain a sharp threshold for a maximal $K_{r+1}$-free graph to have a single pair of \emph{twin} vertices (that is, vertices with identical neighbourhoods). Clearly, this threshold must be lower than the bound in Theorem \ref{thm:SimpleBound}. We consider the following theorem to be the main result of this paper. 

\begin{thm}\label{thm:TwinFreenlogn}
For every $r\geq 2$ there exists a constant $c > 0$ such that every sufficiently large $(r+1)$-saturated graph~$G$ with $e(G) \geq t_{n,r}-cn \log n$ has a pair of twin vertices.
\end{thm}

Note that unlike Theorem~\ref{thm:SimpleBound}, in this case the bounds are similar for all values of~$r$, though the proof is still much shorter in the case $r = 2$ (see Proposition \ref{prop:3SatManyTwins}). As a corollary of Theorem~\ref{thm:TwinFreenlogn}, we obtain a new, simple proof of the aforementioned theorem of Simonovits, formally stated as follows.

\begin{theorem}\label{thm:kChromaticBounded}
For each $r \geq 2$ and each $k \geq r$, there exists $m(k, r)$ such that if $G$ is an extremal $K_{r+1}$-free graph with chromatic number at least~$k$, then $G$ is a blow-up of a graph~$G'$ with $|G'|\leq m(k, r)$. 
\end{theorem}

In other words, for every $r$ and $k$, the sequence of extremal graphs~$G$ for $\omega(G)\leq r$ and $\chi(G)\geq k$ is simple. 

We also discuss some other aspects of stability in the linear sub-regime and prove a number of smaller results. 

It should be said that the corresponding minimal degree (rather then graph size) condition has been extensively studied in a number of papers. This will not be in the scope of our discussion. For a discussion of these results, see, e.g., the survey~\cite{Nik}.

The rest of this paper is organised as follows.  In Section~\ref{se:strongstab}, we prove Theorem~\ref{thm1} and classify the extremal graphs.  In Section~\ref{se:CliqueSat}, we prove Theorems \ref{thm:SimpleBound} and~\ref{thm:TwinFreenlogn}, as well as other results about large $K_{r+1}$-saturated graphs.  In Section~\ref{se:chromturan}, we prove Theorem~\ref{thm:kChromaticBounded}.  We then use this result to determine the size of an extremal $K_{r+1}$-free graph of chromatic number at least~$k$ up to an additive constant.

\section{A new proof of Theorem~\ref{thm1}}\label{se:strongstab}

In this section we prove Theorem~\ref{thm1} and classify the extremal graphs.

First, the following construction shows that the bound in Theorem~\ref{thm1} is tight: take a copy of $T_{n-1,r}$ on partition classes $V_1$, \dots,~$V_r$.
Take a new vertex~$u$ and connect it to each vertex in $V_3$, \dots,~$V_r$ and to one vertex from each of $V_1$~and~$V_2$; call them $v_1$ and $v_2$. Lastly, remove the edge~$v_1v_2$. For $n > 2r$, taking $V_1$ and $V_2$ to be the smallest partition classes, this construction achieves the bound of Theorem~\ref{thm1}. On the other hand, for $r \geq 3$ and $r+3\leq n \leq 2r$ this construction does not work, since the obtained graph will be $r$-colourable. Instead, the extremal construction in this case is achieved by taking $V_1$ and $V_2$ to be the largest partition classes (of size~$2$), for a total of $t_{n,r}-2$ edges.  In each case, we call the resulting graph~$G_{n,r}$.  It is easy to verify that $G_{n,r}$ is $K_{r+1}$-free and is not $r$-colourable.  Finally, for $n\leq r+2$, no graph on $n$ vertices has the required properties. Note that in general the extremal graphs are not unique; we shall discuss this later. 

In the proof of Theorem~\ref{thm1}, we will want to apply the \emph{Zykov symmetrization}, defined as follows. Given a graph~$G$ and independent vertices $u$,~$v \in V(G)$, define $Z_{u,v}(G)$ to be the graph obtained by replacing $u$ with a twin of~$v$. That is, we delete all edges incident to~$u$ and insert edges between $u$ and the neighbours of~$v$ instead. It is easy to see that $\omega(Z_{u,v}(G))=\omega(G\setminus \left\lbrace u \right\rbrace )$ and $\chi(Z_{u,v}(G))=\chi(G\setminus \left\lbrace u \right\rbrace )$, and, as a consequence,
\begin{equation}\label{eq:ZykovOmega}
\omega(G)-1\leq \omega\bigl(Z_{u,v}(G)\bigr)\leq \omega(G)
\end{equation}
and 
\begin{equation}\label{eq:ZykovChi}
\chi(G)-1\leq \chi\bigl(Z_{u,v}(G)\bigr)\leq \chi(G).
\end{equation}

If $\deg(u) < \deg(v)$, replacing $G$ with $Z_{u,v}(G)$, increases $e(G)$, does not increase $\omega(G)$ and decreases $\chi(G)$ by at most~$1$. Similarly, if $\deg(u) = \deg(v)$, then we may apply either $Z_{u,v}$ or $Z_{v,u}$, with the same effect on $\omega(G)$ and $\chi(G)$, while keeping $e(G)$ unchanged. Let us call the Zykov symmetrization~$Z_{u,v}$ \emph{increasing} or an \emph{IZS} if $\deg(u)\leq \deg(v)$. The following lemma is due to Zykov himself \cite{Zy} and leads to his well-known proof of Tur\'{a}n's theorem. For the sake of self-containment we shall recall its short proof here.

\begin{lem}\label{Zykov}
If $\omega(G)\leq r$ then there exists a sequence of increasing Zykov symmetrizations transforming $G$ into a complete $s$-partite graph for some $s\leq r$. 
\end{lem}

\begin{proof}
The transformation $Z_{u,v}$ turns $u$ into a twin of~$v$. Note that `twins' are an equivalence relation, giving rise to \emph{twin classes}. Every pair of twin classes forms either an empty or a complete bipartite graph. In the former case we can repeatedly apply an IZS and merge the two classes into one. We keep doing so until there are no missing edges between vertices from different twin classes, after which the obtained graph $G'$ is complete $s$-partite, where $s$ is the number of twin classes. By~\eqref{eq:ZykovOmega}, $\omega(G')\leq \omega(G)\leq r$, so we must have $s\leq r$, which proves the lemma.
\end{proof}

\begin{proof}[Proof of Theorem \ref{thm1}.]
We will show that any non-$r$-colourable, $K_{r+1}$-free graph~$G$ must have at most as many edges as $G_{n,r}$.

\textbf{Step 1:} We first use the Zykov symmetrization. Recall that the initial graph~$G$ satisfies $\chi(G)>r$, and by~\eqref{eq:ZykovChi} with each IZS the chromatic number decreases by at most~$1$. Therefore, by Lemma~\ref{Zykov}, it suffices to prove that $e(G)\leq h(n,r)$ for every $G$ such that $\omega(G)\leq r$, $\chi(G)=r+1$ and $\chi(Z_{u,v}(G))=r$ for some increasing $Z_{u,v}$. The latter implies that $\chi(G \setminus \left\lbrace u \right\rbrace )=r$, which in turn means that $G$ can be properly $(r+1)$-coloured such that $u$ is the only vertex with its colour.

So from now on let us assume that $V(G)$ can be split into $r+1$ independent sets $V_1, \dots, V_r$, and $\left\lbrace u \right\rbrace$. Observe that for each~$i$, $u$ must have a neighbour~$v_i \in V_i$, for otherwise we could add $u$ to some $V_i$ and obtain an $r$-colouring of~$G$. Furthermore, at least~one edge between some $v_i$ and $v_j$ is missing, for otherwise the $v_i$ and $u$ would form a copy of~$K_{r+1}$.

\textbf{Step 2:} We now apply a series of edge switches as follows. If two neighbours of~$u$ in different partition classes, say $v\in V_i$ and $w\in V_j$, are not adjacent and $u$ has more than~one neighbour in either $V_i$ or $V_j$, say $v'\in V_i$, then we remove the edge~$uv$ and add the edge~$vw$. At this point it does not matter what happens to $\chi(G)$. However, it is crucial that after this switch the resulting graph $\tilde{G}$ is $K_{r+1}$-free.  Indeed, because $G\left[V\setminus\left\lbrace u \right\rbrace \right]$ is $r$-partite, any copy~$F$ of~$K_{r+1}$ in $\tilde{G}$ must contain $u$.  This means that $F$ cannot contain $v$, but in this case a copy of~$K_{r+1}$ would already be present in $G$, a contradiction.

Continue the switches for as long as possible; the procedure will terminate since the degree of~$u$ decreases after each switch. Once no more switches are possible, we end up with a graph~$G'$ such that $u$ has precisely~one neighbour in two of the partition classes, say $v_1\in V_1$ and $v_2\in V_2$, with no edge between $v_1$ and~$v_2$.

\textbf{Step 3:} Now add all missing edges between every $V_i$ and $V_j$ with $i\neq j$ except for $v_1v_2$, and between $u$ and every $V_i$ with $i\geq 3$. The obtained graph is $(r+1)$-chromatic and contains no $K_{r+1}$. Moreover, its size is maximised if the sizes of $V_1, \dots, V_r$ are as close as possible, resulting in $e(G)=h(n,r)$. 
\end{proof}

As was mentioned before, in general the extremal example is not unique. By examining our proof of Theorem \ref{thm1}, the family of extremal examples can be easily characterised (the extremal graphs were also characterised in some of the previous proofs). We let

\begin{equation}\label{eq:sdef}
	s = s(n, r) = \lfloor n/r \rfloor.
\end{equation}
Given $\ell$, $1 \leq \ell \leq s - 1$, let $\Ggraph{r}{\ell}$ be the graph obtained from $G_{n,r}$ as follows: let $W \subset V_1$ with $\abs{W} = \ell$, and, for each $w \in W$, add the edge~$uw$ and remove the edge $v_2w.$  (Note that $G_{n,r} = \Ggraph{r}{1}$.)  If $n = kr + 2$ for some $k \geq 2$, then $V_1$ and $V_2$ have different sizes.  Without loss of generality, let $\abs{V_1} = \abs{V_2} + 1 = k + 1$.  In this case, we may also modify $G_{n,r}$ by connecting $u$ to a set~$W' \subset V_2$ with $\abs{W'} = \ell$ and disconnecting $v_1$ from all elements of~$W'$.  Let $\Gprimegraph{r}{\ell}$ denote the resulting graph and observe that $\Gprimegraph{r}{\ell} \ncong \Ggraph{r}{\ell}$.  Note that for $1 \leq \ell \leq s - 1$, both $\Ggraph{r}{\ell}$ and $\Gprimegraph{r}{\ell}$ are $(r + 1)$-chromatic and $K_{r+1}$-free.

We shall show that these are the only extremal graphs.

\begin{thm}\label{thm2}
Let $r \geq 2$ and $n \geq r + 3$.
Let $h(n, r)$ and $s$ be as in~\eqref{eq:hdef} and~\eqref{eq:sdef}, respectively. If $G$ is a $K_{r+1}$-free graph of order~$n$ and size~$h(n,r)$ that is not $r$-colourable, then there exists some $1 \leq \ell \leq s - 1$ such that $G \cong \Ggraph{r}{\ell}$, or, if $n = kr + 2$ for some $k \geq 2$, then there exists some $1 \leq \ell \leq s - 1$ such that either $G \cong \Ggraph{r}{\ell}$ or $G \cong \Gprimegraph{r}{\ell}$.
\end{thm}

\begin{proof}[Proof of Theorem \ref{thm2}.]
In the proof of Theorem~\ref{thm1}, we provided an algorithm for transforming any non-$r$-colourable $K_{r+1}$-free graph into the graph~$G_{n,r}$ without decreasing its size. In order to classify the extremal graphs, we must examine the procedure in the case where the size of the graph never increases.

Let $G$ be as in the statement of the theorem.  Because $e(G)$ is maximal, we may assume that no edges were added to~$G$ during Step~3 of the algorithm.  Let us therefore consider Step~2 of the construction. Because we did not need to add any edges to~$G$ in Step~3, we may assume that the graph~$G'$ obtained from having done the switches in Step~2 is isomorphic to~$G_{n,r}$.  Now, if $G$ was transformed into $G'$ by switches, then we should be able to perform a series of inverse switches to transform $G'$ into $G$. To perform an inverse switch, we need a pair~$(v,w)$ with $v\in V_i$ and $w\in V_j$, $i\neq j$, such that $uv$,~$vw \in E(G')$ and $uw \notin E(G')$.

By the definition of~$G_{n,r}$, we must have $w \in V_1 \setminus \{v_1\}$ or $w \in V_2 \setminus \{v_2\}$.   So, if we let $v = v_2$ and repeatedly choose $w \in V_1 \setminus \{v_1\}$, then after each inverse switch, we obtain the graph $\Ggraph{r}{\ell}$ for some~$\ell$.  Similarly, if we let $v = v_1$ and repeatedly choose $w \in V_2 \setminus \{v_2\}$, then after each inverse switch, the resulting graph is isomorphic either to~$\Ggraph{r}{\ell}$ or (if $n = kr + 2$ for some $k \geq 2$) to $\Gprimegraph{r}{\ell}$ for some~$\ell$.  This means that the graph remains $K_{r+1}$-free and is not $r$-colourable throughout this process.

Now we show that performing any other inverse switch creates a graph that does not satisfy one of the hypotheses of the theorem.  First, we may not connect $u$ to vertices in both $V_1 \setminus \{v_1\}$ and $V_2 \setminus \{v_2\}$: if, for some $w_1 \in V_1 \setminus \{v_1\}$ and $w_2 \in V_2 \setminus \{v_2\}$, we added the edges $uw_1$ and~$uw_2$ and removed the edges $v_1 w_2$ and~$v_2 w_1$, then for $j = 3, \dots, r$, there exist $v_j \in V_j$ such that $u$, $w_1$, $w_2$, $v_3$, \dots, $v_r$ would induce a copy of~$K_{r+1}$.  Second, we may not connect $u$ to all of~$V_1$ (respectively, to all of $V_2$), because the resulting graph would be $r$-colourable: we could give colour~$2$ to~$u$ and give colour~$1$ to~$v_2$ (respectively, to~$v_1$).

Finally, we may not have $v \in V_j$ for any $j \geq 3$.  Indeed, suppose that for some $w \in V_1 \setminus \{v_1\}$ and $v \in V_3$, say, we added the edge~$uw$ and removed the edge~$wv$.  If $V_3$ contains a vertex~$x$ besides $v$, then there exist vertices $v_i \in V_i$, $i = 4$, \dots, $r$, such that $u$, $w$, $v_2$, $x$, $v_4$, \dots, $v_r$ would induce a copy of~$K_{r+1}$.  (This must be the case if $n \geq 2r + 1$.)  If $\abs{V_3} = 1$, then we must have $n \leq 2r$, and in particular, we must have $\abs{V_1} = \abs{V_2} = 2$.  In this case, the resulting graph would be $r$-colourable: letting $y$ be the vertex of~$V_2$ that is different from $v_2$, we could give colour~$1$ to $u$ and $y$, colour~$2$ to $v$ and~$w$, colour~$3$ to $v_1$ and $v_2$, and, for $i = 4$, \dots, $r$, colour~$i$ to all vertices of~$V_i$.

Hence we may assume that $G$ is transformable into some $\Ggraph{r}{\ell}$ or some $\Gprimegraph{r}{\ell}$ (where $1 \leq \ell \leq s - 1$) by a series of IZS's in Step~1 of the algorithm.  In what follows, we assume that $G$ can be transformed into $G'' \cong \Ggraph{r}{\ell}$ for some~$\ell$; the other case is nearly identical.  Observe that because $e(G)$ is maximal, each IZS leaves the number of edges in the graph unchanged, meaning that at each step we symmetrize two vertices of equal degrees. Again, reversing the procedure, $G''$ is transformable into $G$ by a series of inverse symmetrizations: take two twins $x$ and~$y$, remove $x$ and add a new vertex $x'$ such that $\deg(x')=\deg(y)$, $x' \not \sim y$ and $N(x') \neq N(y)$.  Letting $W = N_{G''}(u) \cap V_1$, it is easy to see that the only twins in $G''$ are pairs of vertices from $W$, pairs of vertices from $V_1 \setminus W$, pairs of vertices from $V_2 \setminus \{v_2\}$, and pairs of vertices from some class $V_i
 $ with $i \geq 3$.  (If $|V_1|=|V_2|=1$, then $v_1$ and $v_2$ are twins, but this contradicts our assumption that $n \geq r + 3$.)

If $x$, $y \in V_1 \setminus W$ and $x'$ is a twin of some $w \in W$, then the resulting graph is isomorphic to~$\Ggraph{r}{\ell + 1}$.  Similarly, if $x$, $y \in W$ and $x'$ is a twin of some $w \in V_1 \setminus W$, then the resulting graph is isomorphic to $\Ggraph{r}{\ell - 1}$.  (In this case, if $\ell = 2$ and $x'$ is a twin of $v_2$, then the resulting graph is isomorphic to either $\Ggraph{r}{2}$ or $\Gprimegraph{r}{2}$.)  Again, at each stage, the graph is $K_{r+1}$-free and is not $r$-colourable.  Any other inverse symmetrization would create a copy of~$K_{r+1}$: in all other cases, either $x'$ has neighbours in all of the $V_j$, all of which are adjacent to one another; or $x'$ is adjacent to~$u$ and to vertices in all but one of the $V_j$, all of which are adjacent to one another and to~$u$; or both.

Thus, our extremal graph~$G$ must be either of the form~$\Ggraph{r}{\ell}$ or of the form~$\Gprimegraph{r}{\ell}$ for some $1 \leq \ell \leq s - 1$.  This completes the proof.
\end{proof}

\section{Clique-saturated graphs}\label{se:CliqueSat}

In this section we shall prove a number of results about stability of $(r+1)$-saturated graphs near the Tur\'{a}n threshold, including Theorems \ref{thm:SimpleBound} and~\ref{thm:TwinFreenlogn}.

We shall make frequent use of the following result of Andr\'asfai, Erd\H{o}s and S\'os~\cite{AES}.

\begin{theorem}\label{thm:AES}
Let $r \geq 2$.  If a graph~$G$ on $n$ vertices is $K_{r+1}$-free and not $r$-colourable, then there exists $v \in V(G)$ such that
\[
\deg(v) \leq \dfrac{3r-4}{3r-1} n.
\]
\end{theorem}

We shall also often use the following immediate corollary of Theorem~\ref{thm:AES}.

\begin{corollary}\label{cor:AEScor}
There exists a function~$g(r,c)$ such that the vertex set of every $K_{r+1}$-free graph~$G$ with $e(G) \geq t_{n,r}-cn$ can be split into a set~$F$ with $|F|\leq g(r,c)$ and an $r$-partite graph on $V\setminus F$.
\end{corollary}

\begin{proof}
Take $F$ to be the set of vertices of degree at most~$\frac{3r-4}{3r-1} n$, which must be of bounded size by the condition $e(G) \geq t_{n,r}-cn$. The remaining vertices induce, by Theorem~\ref{thm:AES}, an $r$-partite graph. 
\end{proof}

\subsection{Finite-size reductions}\label{se:BoundedBlowup}

The unique largest $(r+1)$-saturated graph, the Tur\'an graph~$T_{n,r}$, is a balanced blow-up of~$K_r$.  Moreover, by Theorem \ref{thm1} a $K_{r+1}$-free graph~$G$ that has more than~$t_{n,r}-n/r+1$ edges is $r$-chromatic.  Hence, if $G$ is $(r+1)$-saturated, then all edges between different partition classes must be present, so $G$ is complete $r$-partite (possibly with unbalanced colour classes), i.e.~it is another blow-up of~$K_r$.  It is natural to ask: if we continue to decrease $e(G)$, how long will $G$ remain a blow-up of a finite order graph?  In other words, what is the largest function~$f_r(n)$ such that every $(r + 1)$-saturated graph with at least~$t_{n,r} - f_r(n)$ edges is a blow-up of a graph whose order does not depend on $n$?

We begin by proving Theorem~\ref{thm:SimpleBound} in the case $r = 2$.

\begin{thm}\label{thm3saturated}
For every $c \geq 0$ there exists $m_2(c)$ such that every $3$-saturated graph~$G$ on $n$ vertices with $e(G)> t_{n,2}-cn$ is a blow-up of some (triangle-free) graph~$H$ with $|H|\leq m_2$. 
\end{thm}

\begin{proof}
If $G$ is bipartite, then it must be complete bipartite, and we are done.  If $G$ is not bipartite, then by Corollary~\ref{cor:AEScor} it is composed of a large bipartite graph $G_b=(U,W,E_b)$ and an exceptional vertex set~$V_e$ with $|V_e|\leq g(2,c)$. Now, partition the vertices of $U$~and~$W$ according to their $V_e$-neighbourhoods: for every $X\subset V_e$, define 
\[
U_X:=\left\{u\in U \, : \, N_{V_e}(u)=X\right\},
\]
and $W_X$ analogously. Take any $u\in U$ and $w\in W$, and let $X=N_{V_e}(u)$ and $Y=N_{V_e}(w)$, so that $u\in U_X$ and $w\in W_Y$. If $X \cap Y = \emptyset$, then $u$ and $w$ must be adjacent, since $G$ is $3$-saturated. On the other hand, if $X\cap Y \neq \emptyset$, there can be no edge between $u$ and $w$, as it would create a triangle. Hence, the neighbourhoods of $u$~and~$w$ are completely determined by their $V_e$-neighbourhoods, meaning that two vertices $u_1$,~$u_2 \in U_X$ for any given~$X$ are twins (the same holds in $W$). Since there are at most~$2^{|V_e(G)|}$ possible $V_e$-neighbourhoods, we conclude that $G$ has at most
\[
|V_e(G)|+2\cdot 2^{|V_e(G)|}
\]
twin classes.  Thus, the statement of the theorem holds with $m_2(c)=g(2,c)+2\cdot 2^{g(2,c)}$.
\end{proof}

Note that extremal triangle-free, $(\geq\!k)$-chromatic graphs are in particular $3$-saturated. As was mentioned in the Introduction, it is easy to construct a triangle-free, $(\geq\!k)$-chromatic graph with $t_{n,2} - c_k n$ edges.  
Thus, as an immediate corollary of Theorem~\ref{thm3saturated} we obtain Theorem~\ref{thm:kChromaticBounded} (Simonovits' Theorem) for $r=2$.

\begin{corollary}\label{cor:sublinear}
For each $k \geq 2$ there exists a constant $m(k,2)$ such that if $G$ is an extremal triangle-free, $(\geq\!k)$-chromatic graph on $n$ vertices, then $G$ is a blow-up of a graph $G'$ with $|G'|\leq m(k,2)$.
\end{corollary}

The following construction demonstrates that the bound of  Theorem~\ref{thm3saturated} is sharp in the following sense: given a function~$f(n)$ that tends to infinity (no matter how slowly), there exist $3$-saturated graphs~$G$ with $e(G)=t_{n,2}-nf(n)$, yet with an unbounded number of twin classes. 

\begin{example}\label{ex:3satunbounded}
We may assume that $f(n)<\frac{\log_{2}n}{2}$. Let $S$ be a set of $f(n)$ vertices, let $U$ and $W$ be disjoint sets of $2^{f(n)}$ vertices each, and divide the rest of the vertices equally into two sets $U'$ and~$W'$.  Give different vertices of~$U$ distinct neighbourhoods in $S$, and similarly for vertices in $W$: for each $I \subset S$, let $u_I$ be the vertex in $U$ with $N_S(u_I) = I$, and define $w_I$ similarly. Join $u_I$ and $w_J$ if and only if $I$ and $J$ are disjoint. Finally, add all edges between $U'$ and $W'$, between $U'$ and $W$, and between $U$ and $W'$. It is not hard to see that the resulting graph~$G$ is 3-saturated. Also, $G$ has at least~$2^{f(n) + 1} + f(n)$ distinct neighbourhoods. 

Since $f(n)<\frac{\log_{2}n}{2}$, we obtain
\begin{align*}
e(G)&>|U'||W'|+|U'||W|+|U||W'| > t_{n-f(n),2} - 2^{2f(n)} \\ 
    &> t_{n,2} - \dfrac{n f(n)}{2} - 2^{2f(n)} > t_{n,2} - n f(n),
\end{align*}
as claimed.
\end{example}

Now let us consider the case~$r \geq 3$. Perhaps surprisingly, the analogue of Theorem~\ref{thm3saturated} does not hold here, as the following construction shows.

\begin{example}\label{ex:nosmallblowup}
Let $n \in \N$, let $m = (1/2)\log_{2} n$ and let $M = \binom{m}{m/2}$; note that $M<\sqrt{n}$.  
Take the Tur\'{a}n graph $T_{n-1,r}$ and let $V_1$, \dots,~$V_r$ denote its partition classes. Let $W_1 \subset V_1$, $W_2 \subset V_2$ and $W_3 \subset V_3$ with $\abs{W_1} = M$ and $\abs{W_2} = \abs{W_3} = m$. Introduce a new vertex~$v$ to~$G$ and join it to all of the vertices of the $W_i$ and to all of the vertices of $V_j$ for $j \neq \{1,2,3\}$. Remove all edges between different $W_i$. The resulting graph~$G'$ satisfies 
\[
e(G')\geq t_{n-1,r}-2mM - m^2+\left \lfloor\frac{r-3}{r}(n-1)\right \rfloor+M+2m=t_{n,r}-\dfrac{2n}{r}+o(n).
\]

Now we add a matching between $W_2$ and $W_3$.  Also, for each $w\in W_1$ we select a subset~$U_w \subset W_2$ of size~$m/2$ such that different vertices of~$W_1$ receive distinct subsets. Connect $w$ to~$U_w$ in $W_2$ and to~$W_3\setminus N_{W_3}(U_w)$ in $W_3$.  (Observe that we have added only $m + mM = o(n)$ edges.)

It is easy to check that the obtained graph~$G$ is $(r+1)$-saturated. Moreover, no vertices in $W_1$ are twins, so $G$ has an unbounded number of twin classes.
\end{example}

Given $r \geq 3$, let $c_r$ be the supremum of the numbers~$c$ such that every $(r + 1)$-saturated graph~$G$ with $e(G)> t_{n,r}-cn$ has a bounded number of twin classes.  Observe that Theorem~\ref{thm1} and Example~\ref{ex:nosmallblowup} imply that $1/r \leq c_r \leq 2/r$.  We now show that $c_r=2/r$ holds for all $r\geq 3$, and so complete the proof of Theorem~\ref{thm:SimpleBound}.

\begin{thm}\label{thmrsaturated}
For every $r \geq 3$ and every $\varepsilon>0$ there exists $m_r(\varepsilon)$ such that every $(r+1)$-saturated graph~$G$ with $e(G)> t_{n,r}-(2-\varepsilon)n/r$ is a blow-up of some ($K_{r+1}$-free) graph $H$ with $|H|\leq m_r$. 
\end{thm}

\begin{proof}
To illustrate the main ideas, let us assume first that $r=3$. Let $F$ be the set of all vertices of degree at most $(2/3-\varepsilon/25)n$. By Corollary~\ref{cor:AEScor} $F$ is finite and $G[V\setminus F]$ is $3$-partite; call its partition classes $V_1$, $V_2$ and~$V_3$. By assumption on the size of~$G$ each vertex~$u\in F$ will have at least~$\varepsilon n/10$ neighbours in some $V_i$.  (Otherwise, $\deg(u) \leq 3 \varepsilon n/10+O(1)$, which means that $e(G \setminus \{u\}) > t_{n,3}-(2-\varepsilon/30)n/3>t_{n-1,3}=t_{n,3}-2n/3+O(1)$, a contradiction.)

For each $u\in F$ let $V_i^u$ and $V_j^u$ be those of the $V_i$ to which $u$ has the smallest number of neighbours; call them $u$-\emph{small} classes. By the assumption on $e(G)$, the remaining $u$-\emph{big} class contains at least~$\varepsilon n/10$ neighbours of~$u$. To simplify notation assume that for a given~$u$ we have $V_i^u=V_1$ and $V_j^u=V_2$. By assumption on degrees in $V\setminus F$, no two vertices in $N_{V_1}(u) \cup N_{V_2}(u)$ can be adjacent, because any two such vertices must have at most~$2\varepsilon n/25 + O(1) < \varepsilon n/10$ total non-neighbours in $V_3$, and so must have a common neighbour in $N_{V_3}(u)$. 

Now consider two arbitrary non-adjacent vertices in different partition classes, say $v_1\in V_1$ and $v_2\in V_2$. Let $J=N_{F}(v_1)\cap N_{F}(v_2)$. We show that any two vertices $w_1\in V_1$ and~$w_2\in V_2$ with $N_{F}(w_1)\cap N_{F}(w_2)=J$ are also not adjacent. This will imply that the neighbourhood of every vertex in $V\setminus F$ is determined by its $F$-neighbourhood, resulting in a finite number of twin classes.

Since $G$ is $4$-saturated, adding an edge between $v_1$ and $v_2$ would create a copy~$T$ of~$K_4$. This can happen in two different ways. If $T$ contains two vertices $u_1$,~$u_2\in F$, then adding an edge between $w_1$ and $w_2$ would create a copy of~$K_4$ in $G$, namely $w_1w_2u_1u_2$. If not, then $T$ must contain some $u\in F$ and $v_3\in V_3$. In this case the edges $v_1v_3$ and~$v_2v_3$ are present in $G$, meaning that $V_1$ and $V_2$ must be $u$-small, and, consequently, that $w_1$ and $w_2$ are not adjacent.

For arbitrary $r$ the argument is similar. Let $F$ be the set of all vertices of degree at most $(1-1/r-\varepsilon/(6r+3))n$. Again, $F$ is finite and $G[V\setminus F]$ is $r$-partite; call its partition classes $V_1$, \dots,~$V_r$. By the assumption on $e(G)$, for any set~$X \subset F$ with $k\leq r-2$ vertices, the common neighbourhood of the vertices in $X$ contains at least $\varepsilon n/(3r+1)$ vertices from each of some $r-k-1$ partition classes. Call those partition classes \emph{$X$-big} and the remaining ones \emph{$X$-small}. Again, by the assumption on degrees in $V \setminus F$, if $v_1$ and $v_2$ are vertices in $X$-small partition classes, then they cannot be adjacent, for they must have a common neighbour in $N_{V_i}(X)$ for each $X$-big class~$V_i$, and thus, if they were adjacent, would form a copy of~$K_{r+1}$ consisting of $v_1$, $v_2$, the vertices of~$X$ and their common neighbour in each of the $r-k-1$ $X$-big classes.

Now consider two arbitrary non-adjacent vertices in different partition classes, say $v_1\in V_1$ and $v_2\in V_2$. Let $J=N_{F}(v_1)\cap N_{F}(v_2)$. We show again that any two vertices $w_1\in V_1$ and~$w_2\in V_2$ with $N_{F}(w_1)\cap N_{F}(w_2)=J$ are also not adjacent; this will suffice to show that the number of twin classes is finite.

Since $G$ is $(r+1)$-saturated, adding an edge between $v_1$ and $v_2$ would create a copy of~$K_{r+1}$. Hence, there exist sets $X\subset F$ and $Y\subset V\setminus F$ such that if the edge~$v_1v_2$ were present, then $\{v_1,v_2\}\cup X \cup Y$ would form a copy of $K_{r+1}$. If $|X|=r-1$, then we are done immediately, as the edge between $w_1$ and $w_2$ would, using $X$, also create a copy of $K_{r+1}$. So we may assume that $|X|\leq r-2$ and apply the above split into big and small classes. 

Denote the vertices of~$Y$ by $v_3$, \dots, $v_{r-k+1}$; different vertices must lie in different $V_i$. Since the number of big classes is~$r-k-1$, two of the vertices $v_1$, \dots,~$v_{r-k+1}$ must be in small classes, and the only way this can happen without these vertices being adjacent is if $v_1$ and $v_2$ are in small classes. In this case $w_1$ and $w_2$ will also be in $X$-small classes, and will thereby not be adjacent, as claimed.
\end{proof}

\subsection{Twin-free saturated graphs}\label{se:TwinFree}

In Section \ref{se:BoundedBlowup}, we studied the problem of how many edges ensure that an $(r+1)$-saturated graph is simple (has a bounded number of twin classes). Now we consider a related question: what is the largest number of edges that an $(r+1)$-saturated graph~$G$ can have if \emph{no two} vertices of~$G$ are twins?

We begin by proving Theorem \ref{thm:TwinFreenlogn} in the case $r = 2$.

\begin{prop}\label{prop:3SatManyTwins}
For each $\varepsilon>0$ every sufficiently large 3-saturated graph~$G$ with $e(G) > n^2/4 - (1/10-\varepsilon)n \log_{2} n$ contains a pair of twins.
\end{prop}

\begin{proof}
The argument is similar to the proof of Theorem \ref{thm3saturated}.  Let $G$ be as in the statement of the proposition.  By Theorem~\ref{thm:AES}, we can produce a bipartite subgraph of~$G$ by removing a set~$F$ of $m$ vertices of degree at most~$2n/5 = n/2 - n/10$.  Hence, each vertex removed increases the average degree of the remaining graph by $1/10 + o(1)$.  As $G\setminus F$ is triangle-free, it has average degree at most $\abs{V(G)\setminus F}/2$, and so the bound on $e(G)$ implies that we must have $m < n\log_2 n$, i.e., that $(n-m)/2>2^m$.

Let $V_1$ and $V_2$ be the partition classes of $G[V\setminus F]$. By the same argument as in the proof of Theorem \ref{thm3saturated}, the neighbourhoods of vertices in $V_1$ and $V_2$ are determined by their neighbourhoods in $F$.  Then the bound on $m$ implies that two vertices of the larger partition class will have the same $F$-neighbourhood, which means that they are twins.
\end{proof}

Now we show that Proposition \ref{prop:3SatManyTwins} is best possible up to a constant factor in the $n \log_{2} n$-term.

\begin{example}\label{ex:3satTwinFree}
Fix $m$ and let $n=2m+4\log_2{m}$. We build a graph~$G$ on $n$ vertices as follows. Let $S_1$, $S_2$, $U_1$, $U_2$, $B_1$ and~$B_2$ be pairwise disjoint sets vertex sets with $\abs{B_i}=m$ and $\abs{S_i}=\abs{U_i}=\log_2{m}$ for $i = 1$,~$2$. Add all edges between $S_1$ and~$S_2$, between $U_1$ and~$U_2$ and between $B_1$ and~$B_2$. Give different vertices of~$B_1$ distinct neighbourhoods in $S_2$, and similarly for $B_2$ and $S_1$. Place matchings between $U_1$ and $S_1$ and between $U_2$ and $S_2$. Finally, if $u_1 \in U_1$ and $s_1 \in S_1$ are adjacent, we join $u_1$ to all vertices of~$B_2 \setminus N_{B_2}(s_1)$, and similarly for each $u_2 \in U_2$ and its neighbour~$s_2 \in S_2$. 

It is easy to see that $G$ is twin-free and $3$-saturated.  Furthermore, each vertex in $B:=B_1\cup B_2$ has $m$ neighbours in $B$ and $\log_2{m}$ neighbours outside. Thus
\[
e(G)>m^2+2m\log_2{m}=t_{n,2}-(1+o(1))n\log_{2}n.
\]
\end{example}

Next we prove Theorem~\ref{thm:TwinFreenlogn} for every $r \geq 3$.

\begin{proof}[Proof of Theorem \ref{thm:TwinFreenlogn}.]
Let $G$ be an $(r+1)$-saturated graph on $n$ vertices with no twins. Our aim is to show that, provided $n$ is sufficiently large, $e(G)\leq t_{n,r}-c' n \log n$ for some constant $c'(r)$. 

We may assume that $G$ is not $r$-partite: as observed earlier, if $G$ is $r$-partite, then it must be complete $r$-partite, which implies that every vertex has a twin. Let $F^1$ be the set of low degree vertices, as given by Theorem \ref{thm:AES}. In particular, $G [V \setminus F^1]$ is $r$-partite and $\abs{F^1} \leq c_1\log n$, where $c_1 > 0$ is some constant $c_1$ (otherwise $e(G) \leq t_{n-|F|,r}+\frac{3r-4}{3r-1}n|F|\leq t_{n,r} - c_2 n\log n$ for some constant $c_2$ and we are done). Let $V_1$, \dots,~$V_r$ be the colour classes of $G [V \setminus F^1]$; each of the $V_i$ must have at least~$\frac{n}{2r}$ vertices.

We partition the vertices of~$V_1$ according to their neighbourhood in $F^1$.  Let $W^1$ be an arbitrary partition class and let $A^1 \subset F^1$ denote the common neighbourhood of the vertices of~$W^1$. If $|W_1|>1$, we may assume that $A^1 \neq \emptyset$, for otherwise, if $w \in W^1$ and $v \in V_i$ for some $i \neq 1$, then we may add the edge~$wv$ to~$G$ without creating a copy of~$K_{r+1}$.  However, $G$ is assumed to be $(r+1)$-saturated, which means that all such edges are already present; consequently, all vertices in $W^1$ are twins, a contradiction.

Let $H^1$ be the $(r - 1)$-uniform hypergraph with vertex set~$V_2 \cup \dots \cup V_r$ consisting of all cliques $(v_1, \dots, v_{r-1})$ such that $v_1$, \dots,~$v_{r-1}$ are all adjacent to some $a \in A^1$.  Let $C^1$ be a minimum vertex cover of~$H^1$.  If $\abs{C^1} \geq \log_2 \abs{W^1}$, then we stop.  Otherwise, we set $F^2 = A^1 \cup C^1$ and partition the vertices of~$W^1$ according to their neighbourhoods in $F^2$.

We continue the process inside each partition class as follows. After the $j$th partition, we consider sets~$W^j$, each of which has common neighbourhood~$A^j \subset F^j$.  Once again, we may assume that $A^j$ is non-empty.  We let $H^j$ be the $(r - j)$-uniform hypergraph with vertex set~$V_2 \cup \dots \cup V_r$ consisting of all cliques $(v_1, \dots, v_{r-j})$ such that $v_1$, \dots,~$v_{r-j}$ form a clique of size~$r$ with some $a_1$, \dots,~$a_j \in A^j$.  We let $C^j$ be a minimum vertex cover of~$H^j$.  If $\abs{C^j} < \log_2 \abs{W^j}$, then we set $F^{j+1} = A^j \cup C^j$ and continue.  Otherwise, we stop.

Suppose that $j = r - 2$.  In this case, $H^{r-2}$ is a graph.  Suppose that $\abs{C^{r-2}} < \log_2 \abs{W^{r-2}}$.  We have assumed that none of the vertices in $W^{r-2}$ are twins, but our assumption on $\abs{C^{r-2}}$ means that there must exist $w_1$,~$w_2 \in W^{r-2}$ such that $N_{C^{r-2}}(w_1) = N_{C^{r-2}}(w_2)$.  Hence, there exists $s \notin C^{r-2}$ such that $sw_1 \in E(G)$ but $sw_2 \notin E(G)$.  Because $G$ is $(r + 1)$-saturated, there exists a set~$K$ of $r - 1$ vertices such that if we added the edge~$sw_2$ to~$G$, then $s$, $w_2$, and the vertices of~$K$ would form a copy of~$K_{r+1}$.  Observe that $A^{r-2}$ contains exactly~$r - 2$ vertices of~$K$.  Indeed, if $A^{r-2}$ contained at most~$r - 3$ vertices of~$K$, then there would be an edge of~$H^{r-3}$ disjoint from $F^{r-2}$, contradicting the construction of~$F^{r-2}$.  On the other hand, if $A^{r-2}$ contained all $r - 1$ vertices of~$K$, then $s$, $w_1$, and the vertices of~$K$ would form a copy of~$K_{r+1}$ in $G$, which is again a contradiction.

Let $s'$ be the vertex of~$K$ that is not contained in $A^{r-2}$.  Note that this implies that $s'w_1 \notin E(G)$.  Then our assumption that $N_{C^{r-2}}(w_1) = N_{C^{r-2}}(w_2)$ means that $s' \notin C^{r-2}$.  However, the definition of~$s'$ also implies that $ss'$ is an edge in $H^{r-2}$, which means that $C^{r-2}$ is not a vertex cover of~$H^{r-2}$, a contradiction.

Thus, for some $j \leq r - 2$, we have $\tau(H^j) \geq \log_2 \abs{W^j}$, where $\tau$ denotes the size of a minimum vertex cover.  It is well known that if $H$ is a $t$-uniform hypergraph, then $\tau(H) \leq t\nu(H)$ (simply remove the vertices of a maximum matching), which means that we have $\nu(H^j) \geq c\log_2 \abs{W^j}$.  Let $M$ be a maximum matching of~$H^j$ and let $(v_1, \dots, v_{r-j}) \in M$.  Then for each $w \in W^j$, one of the edges $wv_1$, \dots,~$wv_{r-j}$ is absent from $G$, because by the definition of $H^j$, there exist vertices $a_1$, \dots,~$a_j \in A^j$ that form a clique of size~$r$ with the $v_i$.  Thus, there are at least $c \abs{W^j} \log_2\abs{W^j}$ non-edges between $W^j$ and $V(H^j)$.

The procedure above defines a partition~$\mathcal{W}$ of~$V_1$.  Recalling that $V(H^j) = V_2 \cup \dots \cup V_r$ for each~$j$, we see from the argument above that
\[
e(G) \leq t_{n,r} - c\sum_{W \in \mathcal{W}} \abs{W}\log_2\abs{W} \leq t_{n,r} - c\abs{V_1}\log_2\abs{V_1} \leq t_{n,r} - c'n\log n,
\]
where the second inequality follows from Jensen's inequality.  This completes the proof.
\end{proof}

Observe that in the proof of Theorem~\ref{thm:TwinFreenlogn}, we did not need to assume that $G$ was twin-free, only that it contained a twin-free independent set of size~$cn$ for some $c > 0$.  Thus, Theorem~\ref{thm:TwinFreenlogn} has the following corollary.

\begin{corollary}\label{cor:TwinFreenlognAlt}
For every $\varepsilon > 0$ there exists $\delta > 0$ such that if $G$ is $(r+1)$-saturated and $e(G) \geq t_{n,r} - \delta n\log n$, then at least~$n - \varepsilon n$ vertices of~$G$ have twins.
\end{corollary}

When we apply Corollary~\ref{cor:TwinFreenlognAlt}, we will only use that if $G$ is $(r+1)$-saturated and has a twin-free set of size~$cn$, then $e(G) \leq t_{n,r} - f(n)$, where $f(n)$ tends to infinity with $n$.

Now we show that Theorem~\ref{thm:TwinFreenlogn} is best possible up to the value of the constant~$c$.

\begin{example}\label{ex:r+1SatTwinFree}
The construction is similar to Example \ref{ex:nosmallblowup}.  For $n$~sufficiently large, we construct a twin-free, $(r+1)$-saturated graph on $n$ vertices as follows.  Let $H$ be the disjoint union of $T_{n-r,r}$ and $r$ isolated vertices $u_1$, \dots,~$u_r$.  Let $V_1$, \dots,~$V_r$ denote the colour classes of the copy of~$T_{n-r,r}$.

Let $m$ be a quantity to be defined later and let $M = \binom{m}{m/2}$.  We partition $V_1 \cup \cdots \cup V_r$ into three families of sets $\bigl\{W^{(i)}_1\bigr\}_{i=1}^r$, $\bigl\{W^{(i)}_2\bigr\}_{i=1}^r$ and~$\bigl\{W^{(i)}_3\bigr\}_{i=1}^r$ such that for each $i$, we have $\W{i}{1} \subset V_i$, $\W{i}{2} \subset V_{i+1}$ and $\W{i}{3} \subset V_{i+2}$ (where the addition is modulo~$r$), as well as that $\bigabs{W^{(i)}_1} = M$ and $\bigabs{W^{(i)}_2} = \bigabs{W^{(i)}_3} = m$.
It follows that
\begin{equation}\label{eq:MmRel}
n = r(M + 2m + 1).
\end{equation}
Because $m = o(M)$, \eqref{eq:MmRel} implies that
\begin{equation}\label{eq:MAsymptotic}
M \sim n/r,
\end{equation}
which in turn implies that
\begin{equation}\label{eq:mAsymptotic}
m \sim \log_2 n.
\end{equation}

Now we modify $H$ in order to make it twin-free and maximal $K_{r+1}$-free.  For each~$i$, $i = 1$, \dots,~$r$, we modify $H[\W{i}{1} \cup \W{i}{2} \cup \W{i}{3}]$ as in Example~\ref{ex:nosmallblowup}. 
Then we connect $u_i$ to all vertices of $W^{(i)}_1 \cup W^{(i)}_2 \cup W^{(i)}_3$ and to all vertices of each $V_k$, $k \notin \{i, i+1, i+2\}$ (mod~$r$).
Finally, we greedily add edges among the $u_i$.

Let $G$ denote the resulting graph.  It is easy to check that $G$ is both $(r+1)$-saturated and twin-free.  Moreover, 
\[
e(G) = t_{n-r,r} - r(2Mm + m^2 - Mm - m) + r\Bigl(\dfrac{r - 3}{r}(n - r) + 2m + M \Bigr) + e\bigl(G[\{u_1, \dots, u_r\}]\bigr).
\]
Then \eqref{eq:MmRel},~\eqref{eq:mAsymptotic} and~\eqref{eq:MAsymptotic} imply that
\begin{equation*}\label{eq:EdgeCount}
e(G) = t_{n,r} - r(Mm + m^2 - m) + O(n) = t_{n,r} - rMm + O(n) = t_{n,r} - n\log_2 n + O(n),
\end{equation*}
which is what we wanted to show.
\end{example}

The results of this subsection show that the threshold for the property that an $(r+1)$-saturated graph~$G$ has a pair of twins is $e(G) = t_{n,r} - \Theta(n \log_2 n)$.  We have not attempted to locate the threshold precisely, and leave this as an open problem.

\begin{problem}\label{prob:TwinFreeThreshold}
For $r \geq 2$, determine the supremum of all values~$c$ such that if $n$ is sufficiently large, then every $(r+1)$-saturated graph~$G$ on $n$ vertices with $e(G) \geq t_{n,r} - c n\log_2 n$ has a pair of twins.
\end{problem}

\begin{remark}
Let $c_r$ be the supremum defined in Problem~\ref{prob:TwinFreeThreshold}. Observe that Proposition~\ref{prop:3SatManyTwins} and Example~\ref{ex:3satTwinFree} imply that $1/10 \leq c_2 \leq 1$.  For $r \geq 3$, Example~\ref{ex:r+1SatTwinFree} implies that $c_r \leq 1$, while the lower bound $c_r \geq 1/2r(r-2)$ can be read out of the proof of Theorem~\ref{thm:TwinFreenlogn}.
\end{remark}

The next question asks for the best possible result along the lines of Corollary~\ref{cor:TwinFreenlognAlt}.

\begin{question}
Let $r \geq 2$ and let $c \leq c_r$.  For $n$~sufficiently large, what is the smallest number of vertices with twins that an $(r +1)$-saturated graph on $n$ vertices with at least~$t_{n,r} - cn\log_2 n$ edges may contain?
\end{question}

\subsection{Large complete \texorpdfstring{$r$}{r}-partite subgraphs}\label{se:complete}

In this short section, we consider another way in which an $(r+1)$-saturated graph may be `close' to~$T_{n,r}$, namely, by having a large complete $r$-partite subgraph.  We have shown that if $r \geq 3$ and if $c$ is large enough, then there exist $(r + 1)$-saturated graphs with $t_{n,r} - cn$ edges that are not simple.  However, every $4$-saturated graph with at least this many edges must contain a large complete tripartite subgraph.

\begin{thm}\label{thmK4free}
For every $c>0$ every $4$-saturated graph $G$ with $e(G)> t_{n,3}-cn$ contains a complete tripartite graph on $(1-o(1))n$ vertices.
\end{thm}

For the proof of Theorem \ref{thmK4free} we will need the following two basic facts. 

\begin{lem}\label{le:help1}
If $G$ is triangle-free tripartite graph on $(m,m,m)$ vertices then $e(G)\leq t_{3m,3}-\frac{1}{4}m^2$. 
\end{lem}

\begin{proof}
Suppose otherwise. Let $V_1$, $V_2$,~$V_3$ be the colour classes of~$G$. Since the average degree in $G$ is greater than~$3m/2$, there exists a vertex, say $v \in V_1$, with at least $m/2$ neighbours in each of $V_2$ and $V_3$. As $G$ is triangle-free, all neighbours of $v$ must be independent, resulting in $e(G)\leq t_{3m,3}-\frac{1}{4}m^2$, a contradiction.
\end{proof}

As an immediate consequence, we obtain:

\begin{lem}\label{le:help2}
If $G$ is triangle-free tripartite graph on $(a,b,c)$ vertices, where $a\leq b \leq c$ then $e(G)\leq t_{a+b+c,3}-\frac{1}{4}\left\lfloor\frac{b}{a}\right\rfloor a^2$.
\end{lem}

\begin{proof}
Simply observe that $T_{a+b+c,3}$ contains $\lfloor b/a \rfloor$ edge-disjoint copies of $T_{3a,3}$.
\end{proof}

(Let us note that the maximum size of a triangle-free tripartite graph is studied in detail in~\cite{BSTT}.)

Now we are ready to give the proof of Theorem \ref{thmK4free}.

\begin{proof}[Proof of Theorem~\ref{thmK4free}.]
By Corollary~\ref{cor:AEScor}, we may assume that there exists a set~$V_e\subset V$ with $|V_e|<M=M(c)$ such that $G[V\setminus V_e]$ is tripartite. Let $V_1$, $V_2$ and $V_3$ be the partition classes of~$V\setminus V_e$. For every $v\in V_e$, define $A_v$, $B_v$ and~$C_v$ to be its neighbourhoods in $V_1$, $V_2$ and~$V_3$ such that $|A_v|\leq |B_v|\leq |C_v|$. 

First note that for every $v \in V_e$ the graph $(A_v, B_v, C_v)$ is triangle-free and tripartite. It follows from Lemma \ref{le:help1} that $|A(v)|= O(\sqrt{n})$, for otherwise $e(G) \leq t_{n,3} - \omega(n)$, a contradiction.

Next pick a (large) constant $C$ and split $V_e$ into `small' and 'large' vertices: $V_e=V_s\cup V_\ell$. Put $v\in V_s$ if $|A_v|<C$ and $v\in V_\ell$ otherwise. Notice that if $v\in V_\ell$, then, by Lemma \ref{le:help2} we have $|B_v|\leq c'n$, where $c'= \left ( \frac{4c}{C}+o(1)\right )$. 

Now consider the set 
\[
W:=V\setminus\Biggl (V_e \cup \bigcup_{v\in V_e} A_v \cup \bigcup_{v \in V_\ell} B_v\Biggr).
\]
Putting $W_i:=W\cap V_i$, we have that $|W_i|>n(1/3-c'M)+o(n)$ for each $i$. Let $U:= V_e\cup \bigcup_{v\in V_s} A_v$ and note that $|U|\leq (C+1)M$. We now split the vertices of~$W$ into a finite number of classes according to their neighbourhood in $U$. Let $w_1\in W_1$ and $w_2\in W_2$. We claim that their adjacency depends solely on their neighbourhoods in $U$. If $N_{U}(w_1)\cap N_{U}(w_2)$ is not independent, then $w_1 \not \sim w_2$, for otherwise we would have a copy of~$K_4$. On the other hand, if $N_{U}(w_1)\cap N_{U}(w_2)$ is independent but $w_1$ and $w_2$ are not adjacent, then there exist $v\in V_e$ and $u\in V_3$ such that $w_1$, $w_2$, $v$ and $u$ would form a copy of~$K_4$ if the edge~$w_1w_2$ was added. By definition of~$W$, this can only happen if $v\in V_s$ and $u\in A_v$. But then $u$,~$v \in U$, so $N_{U}(w_1)\cap N_{U}(w_2)$ is not independent, a contradiction. This proves the claim. 

To summarise, $W$ can be split into at most $3\cdot 2^{|U|}\leq 3\cdot 2^{(C+1)M}$ classes such that each pair of classes induces either an empty or a complete bipartite graph. Now consider only those classes in each $W_i$ that are of size at least $2\sqrt{(M+c)n}$. Each pair of them belonging to different $W_i$ must form a complete bipartite graph, otherwise $e(G) \leq t_{n,3} - cn$, a contradiction. Hence, their union forms a complete tripartite graph with at least 
\[
|W_i|- 2^{(C+1)M}\cdot 2\sqrt{(M+c)n}> n(1/3-c'M)+o(n)
\]
vertices in each colour class. Since $C$ was arbitrary and $c'\rightarrow 0$ as $C\rightarrow \infty$, this gives a complete tripartite graph on $(1-o(1))n$ vertices.
\end{proof}

We leave the general case as an open problem.

\begin{problem}
Given $r \geq 4$ and $c > 0$, how large is the largest complete $r$-partite subgraph that an $(r+1)$-saturated graph with at least~$t_{n,r} - cn$ edges is guaranteed to contain?
\end{problem}

\section{Extremal Graphs for the Chromatic Tur\'{a}n Problem}\label{se:chromturan}
In this section we will apply Theorem~\ref{thm:TwinFreenlogn} to give a new proof of Theorem~\ref{thm:kChromaticBounded} for $r\geq 3$ that is much shorter than the original proof by Simonovits~\cite{Sim74}. Recall that for $r=2$ this result was proved in Corollary~\ref{cor:sublinear}. Before we embark on the proof of Theorem~\ref{thm:kChromaticBounded} for arbitrary $r$, we need to introduce some notation.

Let $G$ be a graph with $\omega(G) = r$ and let $C$ be an $r$-clique in $G$ such that the quantity
\[
\sum_{v \in C} \deg(v)
\]
is maximised.  Define
\begin{equation}\label{eq:LambdaGdef}
\Lambda_r(G) := (r-1)\abs{V(G)} - \sum_{v \in C} \deg(v).
\end{equation}
The above expression can also be written as
\begin{equation}\label{eq:LambdaGaltdef}
\Lambda_r(G) = \sum_{v \in V(G)} \bigl(r - 1 - \deg_C(v)\bigr).
\end{equation}

Due to our assumption that $G$ is $K_{r+1}$-free, the right hand side of \eqref{eq:LambdaGaltdef} is non-negative, whereby it has a well-defined minimum over all graphs $G$ with $\omega(G)=r$ and $\chi(G)\geq k$:
\[
\Lambda_r(k):=\min_{\omega(G)=r, \chi(G)\geq k} \Lambda_r(G).
\]
We define $C_{k,r}$ to be the minimal order of a graph realising $\Lambda_{r}(k)$.

Recall from Section~\ref{se:strongstab} that if $u$,~$v \in V(G)$ then the Zykov symmetrization~$Z_{u,v}(G)$ replaces $u$ with a twin of~$v$.  In Section~\ref{se:strongstab}, we required that $u$ and $v$ not be adjacent. Here we extend the notion of $Z_{u,v}$ to the case when $u$ and $v$ are adjacent as follows: we make $u$ a twin of~$v$ and remove the edge~$uv$.

\begin{proof}[Proof of Theorem \ref{thm:kChromaticBounded}.]
Let $G$ be an extremal $K_{r+1}$-free graph on $n$ vertices with chromatic number at least~$k$. Suppose for a contradiction that, as $n\rightarrow \infty$, the number of twin classes in $G$ also tends to infinity. 

As was pointed out in the Introduction, it is immediate that there exists a constant~$c = c(k,r) > 0$ such that $e(G) \geq t_{n,r} - cn$. Thus, by Corollary~\ref{cor:AEScor}, $G$ can be made $r$-partite by removing a set~$F$ of $O_{k,r}(1)$ vertices. Let $V_1$, \dots,~$V_r$ be the partition classes of the remaining subgraph; each of them has to be of size $(1+o(1))n/r$---otherwise, we would have $e(G) \leq t_{n,r} - \omega(n^2)$, a contradiction.

Note that $G$ is in particular $(r+1)$-saturated. Since $e(G)$ exceeds the bound of Corollary ~\ref{cor:TwinFreenlognAlt}, for each $i$, a set $T_i\subset V_i$ of size $t_i = \abs{T_i}=(1+o(1))n/r$ will have twins in $V_i$.

\begin{claim}\label{cl:balanced}
We may assume that each $T_i$ forms a single twin class, that each pair~$(T_i, T_j)$ induces a complete bipartite graph, and that there exists $c > 0$ such that for all $i$ and $j$,
\begin{equation}\label{eq:TwinClassSizes}
\sum_{i=1}^r |V_i \setminus T_i| \geq c\abs{t_i - t_j}.
\end{equation}
\end{claim}

\begin{proof}[Proof of Claim \ref{cl:balanced}.]

First, it is easy to see that if $u$,~$v \in T_i$, then $\deg(u) = \deg(v)$: if not, then either $Z_{u,v}(G)$ or $Z_{v,u}(G)$ is $K_{r+1}$-free, is $(\geq\!k)$-chromatic and has strictly more edges than $G$, a contradiction.  So, by applying Zykov symmetrization within each $V_i$, we may assume that each $T_i$ is a single twin class.  (Note that because we only symmetrize vertices that have twins, this process will not decrease $\chi(G)$.)

Therefore, for each $i$ and $j$, $G[T_i \cup T_j]$ is either empty or complete bipartite.  Because of this and the fact that $t_i = (1+o(1))n/r$, for each $i$ and $j$, we must have $E(T_i, T_j) \neq \emptyset$: otherwise, $e(G)  \leq t_{n,r} - \omega(n^2)$ a contradiction. Hence, for each $i$ and $j$, $G[T_i \cup T_j]$ must be complete bipartite.

It remains to show that~\eqref{eq:TwinClassSizes} holds. Since, by assumption, the number of twin classes in $G$ is unbounded, so must be the left hand side of ~\eqref{eq:TwinClassSizes}. Hence, if~\eqref{eq:TwinClassSizes} does not hold, then there exist $i$ and $j$ such that $t_i - t_j = f(n)$, where $f(n)$ tends to infinity with $n$. In this case we have $\sum_{i=1}^r |V_i \setminus T_i| = o(f(n))$. Then the fact that all edges are present between different $T_i$ implies that if $v_i \in T_i$ and $v_j \in T_j$, then $\deg(v_j)-\deg(v_i) \geq (1-o(1))f(n)$.  Therefore, if we replace $v_i$ with a twin of~$v_j$, then we obtain a graph with strictly more edges than $G$ that is $K_{r+1}$-free and (because we have symmetrized vertices that have twins) is still $(\geq\!k)$-chromatic, which is a contradiction. This proves the claim.
\end{proof}

Let
\[
T = \max_i t_i \qquad \text{and} \qquad t = \min_i t_i.
\]
Let $G'$ be the graph obtained by identifying $t$ of the vertices of each $T_i$. Then $G'$ has $n' := n - r(t - 1)$ vertices; moreover, $n'$ is at least the number of twin classes in $G$, so by assumption $n'$ tends to infinity with $n$.  

We want to show that $G'$ has a large twin-free independent set.  Indeed, observe that $n' \leq r(T - t) + \sum_{i=1}^r |V_i \setminus T_i|$.  It follows from Claim~\ref{cl:balanced} that there exists $c' > 0$ such that for some~$i$, we have $|V_i \setminus T_i| \geq c'n'$.  The set~$V_i \setminus T_i$ is twin-free by definition, so, if $n$ (and hence $n'$) is large enough, Corollary~\ref{cor:TwinFreenlognAlt} implies that
\[
e(G') \leq t_{n', r} - Cn'
\]
for some large constant~$C$.

Let $H''$ be a $K_{r+1}$-free, $k$-chromatic graph on $\ell=C_{k,r}$ vertices such that $\Lambda_r(H'') = \Lambda_r(k)$ (recall that $C_{k,r}$ was defined as the smallest order of such a graph). Let $K \subset V(H'')$ be a clique that achieves the value of~$\Lambda_r(H'')$ and let $H'$ be the graph on $n'$ vertices obtained from $H''$ by blowing up each vertex of~$K$ by a factor of~$(n' - \ell)/r + 1$.  Observe that
\[
e(H') \geq t_{n',r} - cn'
\]
for some small constant~$c$ and that
\[
\Lambda_r(H') = \Lambda_r(H'') = \Lambda_r(k),
\]
where the value of $\Lambda_r(H')=\Lambda_r(k)$ is realised in $H'$ by the same clique $K$. 
In particular, we have
\begin{equation}\label{eq:MoreEdges}
e(H') > e(G')
\end{equation}
and
\begin{equation}\label{eq:BetterLambda}
\Lambda_r(H') \leq \Lambda_r(G').
\end{equation}

Let $v_i \in V(G')$ denote the vertex obtained by identifying the $t$ vertices of~$T_i$.  Thus, we obtain $G$ from $G'$ by blowing up each $v_i$ by a factor of~$t$.  Let $H$ be the graph on $n$ vertices obtained by blowing up each $v \in K\subset H'$ by a factor of~$t$.  Denote $C = \{v_1, \dots, v_r\} \subset V(G')$.  It follows from \eqref{eq:LambdaGdef}, \eqref{eq:MoreEdges} and~\eqref{eq:BetterLambda} that
\begin{align*}
e(G) &= e(G') + (t - 1) \sum_{v \in V(G')} \deg_C(v) + (t - 1)^2\dbinom{r}{2} \\
&\leq e(G') +(t - 1) \bigl((r - 1)n' - \Lambda_r(G')\bigr) + (t - 1)^2\dbinom{r}{2} \\
&< e(H') + (t - 1) \bigl((r - 1)n' - \Lambda_r(H')\bigr) + (t - 1)^2\dbinom{r}{2} \\
&=e(H),
\end{align*}
contradicting the extremality of~$G$.
\end{proof}

\begin{remark}\label{re:ChromaticConditions}
As noted in the Introduction, Theorem~\ref{thm:kChromaticBounded} was first proved by Simonovits~\cite{Sim74} in a more general setting of `chromatic conditions'---properties that are natural generalizations of statements such as `$G$ has chromatic number at least~$k$'.  (For a precise statement, see~\cite[Definition 1.5]{Sim74}.)  It is not hard to verify that our proof of Theorem~\ref{thm:kChromaticBounded} extends to $K_{r+1}$-free graphs that satisfy these more general conditions. We note that Simonovits's results also extend to a larger class of forbidden subgraphs than the class of complete graphs.
\end{remark}

With Theorem \ref{thm:kChromaticBounded} at our disposal, it is a straightforward exercise to determine the correct asymptotics of the extremal numbers for $K_{r+1}$-free, $(\geq\!k)$-chromatic graphs. In fact, the coefficient in the linear term can be conveniently described using the quantity~$\Lambda_r(k)$. 

\begin{thm}\label{thm:Kr+1free}
Let $r \geq 2$ and let $k \geq r + 1$.  If $G$ is a $K_{r+1}$-free, $(\geq\!k)$-chromatic graph maximising $e(G)$ over all such graphs of order $n$, then 
\begin{equation}\label{eq:MaxEdgesKr+1Free}
 e(G)= t_{n,r} - \dfrac{\Lambda_r(k)}{r}\cdot n + O_{k,r}(1).
\end{equation}
\end{thm}

Given a graph $H$ of order $\ell$ with $\omega(H)=r$, the following lemma tells us which of its blow-ups to order~$n$ maximises $e(G)$.

\begin{lem}\label{lemlagrangian}
Let $H$ be a graph on $\ell$ vertices with $\omega(H)=r$ and let $G$ be a blow-up of $H$ with $|V(G)|=n$. For large $n$, $e(G)$ is maximised up to $O(1)$ by letting $C$ be an $r$-clique in $H$ for which the quantity
\begin{equation*}\label{eq:degreesum}  
\sum_{i \in C} \deg(i)
\end{equation*}
is maximised and by blowing up each $i \in C$ by a factor of $(n-\ell)/r+1$. 
\end{lem}

The proof of the lemma is a variant of the proof Tur\'an's theorem due to Motzkin and Straus~\cite{MS}, so we shall only give a sketch of the argument.

\begin{proof}[Sketch of proof.]
It is easy to see that up to a $O(1)$ error term, the problem of maximising $e(G)$ is equivalent to the problem of determining
\[
\max\Bigl\{2\sum_{ij \in E(H)} x_i x_j \, : \, \sum_{i \in V(H)} x_i = n, x_i \geq 1 \text{ for all $i$}\Bigr\}.
\]
Letting $y_i = x_i - 1$, this is equivalent to determining
\begin{equation}\label{eq:maxprogram}
\max\Bigl\{2\Bigl(\sum_{ij \in E(H)} y_i y_j + \sum_{i \in V(H)} y_i \deg(i) + e(H)\Bigr) \, : \, \sum_{i \in V(H)} y_i = n - \ell, y_i \geq 0 \text{ for all $i$}\Bigr\}.
\end{equation}
Thus, given~$\yvec \in \R^{\ell}$, we define
\[
f(\yvec) = 2\sum_{ij \in E(H)} y_i y_j + 2\sum_{i \in V(H)} y_i \deg(i) + 2e(H).
\]

By compactness, the maximum in~\eqref{eq:maxprogram} is achieved.  Moreover, by arguing as in~\cite{MS}, one can show that if $\yvec$ achieves the maximum in~\eqref{eq:maxprogram} and $C = \{i : y_i > 0\}$, then we may assume that $C$ is a clique.  A simple calculation shows that $f(\yvec)$ is maximised when the quantity
\begin{equation}\label{eq:charfunction}
\sum_{i \in C} \bigl(y_i^2 - 2y_i \deg(i)\bigr)
\end{equation}
is minimised.  Recalling that $\sum_{i \in C} y_i = n - \ell$, we find that~\eqref{eq:charfunction} is minimised when
\[
y_i - \deg(i) = y_j - \deg(j)
\]
for all $i$ and $j$, which shows that the $y_i$ must differ by constants (with respect to~$n$). Shifting constant weights in order to make all weights in $C$ equal will then change $e(G)$ only by a constant.
\end{proof}

The next result follows from Lemma~\ref{lemlagrangian} by straightforward calculations.

\begin{cor}\label{cor:blowupbound}
Let $H$ be a graph on $\ell$ vertices with $\omega(H) = r$.  If a graph~$G$ of order~$n$ is a blow-up of~$H$ with the maximum number of edges, then
\begin{equation*}\label{eq:maxblowup}
e(G)=t_{n,r}-\dfrac{\Lambda_r(H)}{r}\cdot n + O(1).
\end{equation*}
\end{cor}

It is now a short step to complete the proof of Theorem~\ref{thm:Kr+1free}.

\begin{proof}[Proof of Theorem \ref{thm:Kr+1free}]
It follows from Corollary~\ref{cor:sublinear} (for $r = 2$) and Theorem~\ref{thm:kChromaticBounded} (in general) that $G$ is a blow-up of a fixed-size graph.  The result then follows from Corollary~\ref{cor:blowupbound}.
\end{proof}

\begin{remark}\label{re:finitary}
With additional work, one can show that there exists a finite-time algorithm such that for each $r$ and $k$,  it is possible to determine the extremal size of a $K_{r+1}$-free, ($\geq\!k$)-chromatic graph exactly.  We omit the details.  We also note that the fact that such an algorithm exists for $r = 2$ was first observed by Simonovits~\cite{Sim69}.
\end{remark}

\begin{remark}\label{re:ChromaticConditionsExtremal}
Theorem~\ref{thm:Kr+1free} also holds for graphs that satisfy the chromatic conditions discussed in Remark~\ref{re:ChromaticConditions}.  (If $\mathcal{A}$ is a chromatic condition, then~\eqref{eq:MaxEdgesKr+1Free} holds with $\Lambda_r(k)$ replaced by $\Lambda_r(\mathcal{A})$, which we define to be the minimum value of~$\Lambda_r(H)$ over all $K_{r+1}$-free graphs~$H$ satisfying $\mathcal{A}$.)
\end{remark}

Recall that Theorem~\ref{thm1} gives the largest number of edges in a $K_{r+1}$-free, ($\geq\!k$)-chromatic graph where $k=r+1$. As an application of Theorem~\ref{thm:Kr+1free}, we will establish the analogous result for $k=r+2$, up to a $O(1)$ error term.

\begin{thm}\label{thm:rplus2}
Let $G$ be a $K_{r+1}$-free, $(\geq\!r + 2)$-chromatic graph that maximises $e(G)$ over all such graphs of order $n$. If $r = 2$, then 
\[
e(G)= t_{n,2} - \dfrac{3n}{2} + O(1),
\]
and if $r \geq 3$, then
\[
e(G)= t_{n,r} - \dfrac{2n}{r} + O_{r}(1).
\]
\end{thm}

By Theorem~\ref{thm:Kr+1free}, in order to prove Theorem~\ref{thm:rplus2}, it is enough to determine $\Lambda_r(r+2)$ for all $r \geq 2$.

\begin{lemma}\label{le:1stLambdabound}
We have $\Lambda_2(4) = 3$.
\end{lemma}

Let us note that Lemma~\ref{le:1stLambdabound} was also proved in~\cite{Sim69}.

\begin{proof}
Let $H$ be a triangle-free graph and let $v$ and $w$ be adjacent vertices of~$H$ such that $\abs{H} - \deg(v) - \deg(w) = \Lambda_2(H)$.  Because $v$ and $w$ are adjacent and $H$ is triangle-free, the quantity~$\abs{H} - \deg(v) - \deg(w)$ is exactly the number of common non-neighbours of $v$~and~$w$.  Let $S$ denote the set of common non-neighbours of $v$~and~$w$.  We claim that if $\abs{S} \leq 2$, then $H$ is $3$-colourable.

First, suppose that $\abs{S} = 1$ and let $S = \{x\}$.  Then, because $N(v)$ and $N(w)$ are independent sets, we may give colour~$1$ to each vertex in~$N(w)$ (including $v$), colour~$2$ to each vertex in $N(v)$ (including $w$), and colour~$3$ to~$x$.

If $\abs{S} = 2$ and $S$ consists of two independent vertices $x$ and~$y$, then we may give colour~$3$ to both of them.  If $x$ and $y$ are adjacent, then we modify the colouring above: we give colour~$3$ to~$x$, colour~$1$ to~$y$ and colour~$3$ to all vertices of~$N(y) \cap N(w)$.  Because $x$ and $y$ are adjacent, they have no common neighbours, and so we have a proper $3$-colouring of~$H$.

It follows that $\Lambda_2(4) \geq 3$.
Finally, $\Lambda_2(4) = 3$ is realised when $H$ is the Gr\"otzsch graph and $v$ and $w$ are adjacent vertices of degree~$4$.
\end{proof}

Now we establish a relation between extremal numbers for different values of~$r$.

\begin{lem}\label{le:monotone}
We have $\Lambda_r(k)\leq \Lambda_{r-1}(k-1).$
\end{lem}

\begin{proof}
Take $H$ that realises $\Lambda_{r-1}(k-1)$ and add a new vertex $u$ adjacent to every vertex of $H$.  Then $\Lambda_r(H \cup \{u\}) = \Lambda_{r-1}(H) = \Lambda_{r-1}(k-1)$, and the result follows.
\end{proof}

Next, we give a lower bound on $\Lambda_r(k)$.

\begin{lem}\label{le:Lambdalb}
We have $\Lambda_r(k) \geq k - r$.
\end{lem}

\begin{proof}
Let $H$ be a graph with $\omega(H) = r$ and $\chi(H) \geq k$ and let $C \subset V(H)$ be an $r$-clique that achieves the value of~$\Lambda_r(H)$.  Let $S = \{v \notin C : d_C(v) = r - 1\}$.  We observe that $H[C \cup S]$ is $r$-colourable: after properly colouring $C$, give each $v\in S$ the colour of its non-neighbour in $C$, and observe that because $\omega(H) = r$, if $u$,~$v \in S$ have the same non-neighbour in $C$, then they are independent.  Thus, our assumption that $\chi(H) \geq k$ means that $H$ contains at least~$k - r$ vertices not in $C \cup S$, and our assumption that $H$ is $K_{r+1}$-free means that each such vertex is adjacent to at most~$r - 2$ vertices of~$C$.  It follows from~\eqref{eq:LambdaGaltdef} that each such vertex contributes at least~1 to~$\Lambda_r(H)$, which proves the lemma.
\end{proof}

In order to prove Theorem~\ref{thm:rplus2}, it remains to compute $\Lambda_r(r + 2)$ for all $r \geq 3$.  However, it turns out to be enough to determine $\Lambda_3(5)$, which we now do.

\begin{lemma}\label{le:Lambda35}
We have $\Lambda_3(5) = 2$.
\end{lemma}

\begin{proof}
Lemma~\ref{le:Lambdalb} implies that $\Lambda_3(5) \geq 2$.  To show that equality holds, we define a graph~$H$ as follows.  Let $v_1$, $v_2$ and~$v_3$ be the vertices of a triangle.  Let $a_{12}$ and $b_{12}$ be adjacent to $v_1$~and~$v_2$, let $b_{23}$ and $c_{23}$ be adjacent to $v_2$~and~$v_3$, and let $a_{13}$, $b_{13}$ and~$c_{13}$ be adjacent to $v_1$~and~$v_3$.  Let $x$ be adjacent to~$v_1$ and to both of the~$a_{ij}$,  let $y$ be adjacent to~$v_3$ and to both of the~$c_{ij}$, and let $x$ and $y$ be adjacent to each other and to all of the $b_{ij}$.  Finally, if $(i, j) \neq (k, \ell)$, let $a_{ij}$ be adjacent to~$c_{k\ell}$.

By inspection, $H$ is $K_4$-free and $\Lambda_3(H) = 2$.  We will show that $H$ is not $4$-colourable.  Consider a proper colouring of~$V(H)$.  For each $i$, we give colour~$i$ to~$v_i$.  We will show that no matter what colours we give to $x$ and~$y$, some vertex of~$H$ must receive colour~5.  If we give colour~3 to~$x$ and colour~1 to~$y$, then neither $a_{12}$ nor $c_{23}$ can receive colours 1,~2 or~3, which means that one of them must receive colour~5.  In the same way, if we give colour~3 to~$x$ and colour~2 to~$y$, then either $a_{12}$ or~$c_{13}$ must receive colour~5, and if we give colour~2 to~$x$ and colour~1 to~$y$, then either $a_{13}$ or~$c_{23}$ must receive colour~5.  Finally, if we give colour~4 to either $x$ or $y$, then no matter what colour we give to the other, some~$b_{ij}$ must receive colour~5.

It follows that $\Lambda_3(5) = \Lambda_3(H) = 2$, as claimed.
\end{proof}

\begin{proof}[Proof of Theorem \ref{thm:rplus2}.]
The result for $r = 2$ follows from Lemma~\ref{le:1stLambdabound} and Theorem~\ref{thm:Kr+1free}.  If $r \geq 3$, then Lemmas \ref{le:monotone},~\ref{le:Lambdalb} and~\ref{le:Lambda35} imply that $\Lambda_r(r + 2) = 2$.  The result then follows from Theorem~\ref{thm:Kr+1free}.
\end{proof}

Let us note that it is possible to determine $\Lambda_2(k)$ for other small values of~$k$, as in the following proposition.

\begin{proposition}\label{prop:2ndLambdabound}
We have $\Lambda_2(5) = 6$.
\end{proposition}

\begin{proof}[Sketch of proof.]
Let $H$ be a triangle-free graph and let $S$ be as in the proof of Lemma~\ref{le:1stLambdabound}.  It is not hard to show that if $\abs{S} \leq 5$, then $H$ is $4$-colourable.

Let $F = C_5 \cup K_1$.  We construct a graph~$H$ consisting of $F$, two additional vertices $v$ and~$w$, and, for each independent set $I \subset V(F)$, two more vertices $v_I$ and~$w_I$.  We join each $v_I$ to~$v$, to all of the vertices in $I$, and to each $w_J$ for which $I \cap J = \emptyset$, and do likewise for the $w_I$.  We also join $v$ and $w$.  Observe that $H$ is triangle-free and that $\Lambda_2(H) = 6$.  It remains to show that $H$ is not $4$-colourable.  Up to relabeling of colours, $F$ admits two proper $3$-colourings and four proper $4$-colourings.  It is not hard to show that for each such colouring, there is a set of three colours, each of which we must give to some $v_I$ and to some $w_J$.  Thus, either $v$ or $w$ must receive colour~$5$, and $\Lambda_2(5) = \Lambda_2(H) = 6$, as claimed.
\end{proof}

\begin{remark}\label{re:TriangleFreeAsymptotics}
Using results from Ramsey theory, it is also possible to give good bounds on $\Lambda_2(k)$ for large $k$.  (We note that this connection was also observed in~\cite{Sim69}.)

First, letting $f_2(k)$ denote the minimum order of a triangle-free graph with chromatic number at least~$k$, it is not hard to show that there exist constants $c_1$,~$c_2 > 0$ such that 
\begin{equation}\label{eq:trianglefreeorder}
c_1 k^2 \log k \leq f_2(k) \leq c_2 k^2 \log k.
\end{equation}
Ajtai, Koml\'os and Szemer\'edi~\cite{AKS80,AKS81} and Kim~\cite{Kim} proved that there exist constants $C_1$,~$C_2 > 0$ such that for every $t \geq 2$, the Ramsey number~$R(3, t)$ satisfies
\begin{equation*}\label{eq:ramsey}
C_1\dfrac{t^2}{\log t} \leq R(3, t) \leq C_2\dfrac{t^2}{\log t}.
\end{equation*}
In other words, every triangle-free graph on $n$ vertices has an independent set of size at least~$c_3 \sqrt{n \log n}$, while there exist triangle-free graphs on $n$ vertices with no independent set of size more than~$c_4 \sqrt{n \log n}$.
The upper bound in~\eqref{eq:trianglefreeorder} then follows from the inequality
\[
\abs{V(G)} \leq \alpha(G) \cdot \chi(G),
\]
while the lower bound can be derived using a greedy algorithm: we repeatedly colour the largest independent set with a single colour and remove it from the graph, as the resulting graph is still triangle-free (see \cite[pp. 124--125]{JT} for details). 

It follows from~\eqref{eq:trianglefreeorder} that there exist constants $c_5$,~$c_6 > 0$ such that 
\begin{equation*}\label{eq:Lambdaasymptotics}
c_5 k^2 \log k \leq \Lambda_2(k) \leq c_6 k^2 \log k.
\end{equation*}
To see this, let $H$ be a triangle-free graph with chromatic number~$k$ on at most~$c_6 k ^2 \log k$ vertices. Then
\[
\Lambda_2(k)\leq \Lambda_2(H)\leq |V(H)|\leq c_6 k ^2 \log k.
\] 
On the other hand, given a triangle-free $k$-chromatic graph $H$ with vertices $v$ and~$w$ realising $\Lambda_2(H)$, put $F=H \setminus \left(N(v)\cup N(w)\right)$. Since $\chi(F)\geq k-2$, \eqref{eq:trianglefreeorder} implies that
\[
\Lambda_2(H)=|V(F)|\geq c (k-2)^2\log (k-2) \geq c_5 k^2 \log k,
\]
for a suitably chosen constant $c_5$.  Since this holds for every $H$, we conclude that $\Lambda_2(k) \geq c_5 k^2 \log k$, as claimed.

It is possible to derive asymptotic bounds on $\Lambda_r(k)$ for fixed $r \geq 3$ in a similar fashion. That said, the existing bounds on $R(s, t)$ for fixed $s \geq 4$ are too far apart to give matching upper and lower bounds when $r \geq 3$.
\end{remark}

\section{Acknowledgments}

We are grateful to Mikl\'os Simonovits for helpful conversations and for providing us with a copy of~\cite{Sim69}.

\bibliographystyle{amsplain}
\bibliography{ChromTuranBib}

\end{document}